\documentclass{amsart}

\usepackage{amssymb}
\usepackage{amsthm}  
\usepackage{amsmath} 
\usepackage{amscd} 
\usepackage[all]{xypic}
\usepackage{url}
\usepackage{color}

\newcommand{\Ann}{\mathfrak{ann}}
\newcommand{\fp}{\mathfrak p}
\newcommand{\fk}{\mathfrak k}
\newcommand{\fh}{\mathfrak h}
\newcommand{\fl}{\mathfrak l}
\newcommand{\fq}{\mathfrak q}
\newcommand{\fm}{\mathfrak m}
\newcommand{\Ad}{{\rm Ad}}
\newcommand{\ad}{{\rm ad}}

\newcommand{\id}{{\rm id}}
\newcommand{\na}{\nabla}
\newcommand{\U}{\Upsilon}

\newcommand{\Rho}{{\mbox{\sf P}}}
\newcommand{\R}{\mathbb{R}}
\newcommand{\C}{\mathbb{C}}

\newtheorem{prop*}{Proposition}

\newtheorem{thm*}{Theorem}

\newtheorem{lemma*}{Lemma}

\newtheorem{cor*}{Corollary}

\newtheorem{rem*}{Remark}
\theoremstyle{definition}
\newtheorem{def*}{Definition}

\theoremstyle{remark}
\newtheorem{exam}{Example}

\begin{document}
\title{On symmetric CR geometries of hypersurface type}
\author{Jan Gregorovi\v c and Lenka Zalabov\' a}
\address{J.G. Faculty of Mathematics, University of Vienna, Oskar Morgenstern Platz 1, 1090 Wien, Austria; L.Z. Institute of Mathematics, Faculty of Science, University of South Bohemia, Brani\v sovsk\' a 1760, \v Cesk\' e Bud\v ejovice, 370 05, Czech Republic and Department of Mathematics and Statistics, Faculty of Science, Masaryk University, Kotl\' a\v rsk\' a 2, Brno, 611 37, Czech Republic}
\email{jan.gregorovic@seznam.cz, lzalabova@gmail.com}
\thanks{First author supported by the project P29468 of the Austrian Science Fund (FWF). Second author supported by the grant 17-01171S of the Czech Science Foundation (GA\v CR)}

\maketitle
\begin{abstract} 
We study non--degenerate CR geometries of hypersurface type that are symmetric in the sense that, at each point, there is a CR transformation reversing the CR distribution at that point. We show that such geometries are either flat or homogeneous. We show that non--flat non--degenerate symmetric CR geometries of hypersurface type are covered by CR geometries with a compatible pseudo--Riemannian metric preserved by all symmetries. We construct examples of simply connected flat non--degenerate symmetric CR geometries of hypersurface type that do not carry a pseudo--Riemannian metric compatible with the symmetries.
\end{abstract}

\section{Introduction} \label{prvni-cast}

In \cite{KZ}, Kaup and Zaitsev generalized Riemannian symmetric spaces to the setting of CR geometries, i.e., smooth manifolds with so--called CR distribution endowed with a complex structure. They consider a Riemannian metric compatible with CR geometry in the sense that the Riemannian metric is preserved by the complex multiplication on the CR distribution. Such manifold is symmetric in the sense of \cite{KZ} if, at each point, there is an isometric CR transformation that preserves the point and which, at that point, acts as $-\id$ on the CR distribution  \cite[Definition 3.5.]{KZ}. They show that such isometric CR transformations are uniquely determined by the tangent action on the CR distribution \cite[Theorem 3.3]{KZ}. They also show that such CR geometries are homogeneous \cite[Proposition 3.6]{KZ}. In fact, these CR geometries may be considered as reflexion spaces in the sense of \cite{L1}. In  \cite{AMN}, the authors study these CR geometries in the setting of so--called CR algebras.
 
We studied in \cite{my-sigma} filtered geometric structures that carry an automorphism at each point that fixes the point and acts as $-$id on a distinguished part of the filtration at the point. Let us point out that the non--degenerate CR geometries of hypersurface type, i.e., those with CR distribution of codimension $1$, are among these geometries. We answered the question whether these filtered geometries are homogeneous and can be considered as reflexion spaces. However, our result \cite[Theorem 5.7.]{my-sigma} holds under weaker conditions than the result of \cite{KZ} for non--degenerate CR geometries of hypersurface type. In particular, the sufficient condition for such non--degenerate CR geometry of hypersurface type to be homogeneous is that it  is non--flat at one point.

In this paper, we study the case of non--degenerate CR geometries of hypersurface type in more detailed way. We consider point preserving CR transformations which, at that point, induce $-$id on the CR distribution. We say that a non--degenerate CR geometry of hypersurface type is symmetric (in our sense) if there exists a symmetry at each point, see  Definition \ref{def1}. In particular, our definition does not require the existence of a metric compatible with the CR geometry. We adapt and significantly improve general results of \cite{my-TG,my-dga,my-sigma} for our particular class of CR geometries. Let us emphasize that every non--degenerate CR geometry of hypersurface type that is symmetric in the sense of \cite{KZ} is symmetric (in our sense).

Let us say that \cite[Theorem 5.7.]{my-sigma} is formulated in the general setting of parabolic geometries. We provide here the particular results of this theorem for CR geometries. We also provide new direct proofs, because we will need the presented ideas to explain new results, see Lemmas \ref{dve-symetrie}, \ref{involutivity} and Propositions \ref{hladky-sys}, \ref{thm1}.  This allows us to compare our results with results of \cite{KZ} and \cite{AMN}.

We prove in  Theorem \ref{thm3} that non--flat non--degenerate CR geometries of hypersurface type that are symmetric (in our sense) are covered by symmetric non--degenerate CR geometries of hypersurface type that carry a pseudo--Riemannian metric compatible with the CR geometry that is preserved by all our symmetries. In the Riemannian signature, these coverings are symmetric in the sense of \cite{KZ}, see Theorem \ref{thm4}. Moreover, we show in Theorem \ref{embed} that it is always possible to embed the CR geometry on these coverings into a complex manifold. In the Riemannian signature, this embedding is provided by a different construction than the one in \cite[Proposition 7.3]{KZ}.

Finally, we construct examples of non--homogeneous symmetric (in our sense)  flat non--degenerate CR geometries of hypersurface type. These examples do not admit a pseudo--Riemannian metric that would be preserved by some symmetry at each point and in particular, they are not symmetric in the sense of \cite{KZ}. We also discuss examples of homogeneous CR geometries on orbits of real forms in complex flag manifolds. In particular, we show that there are homogeneous CR geometries which are locally symmetric but not globally symmetric. In fact, Theorem \ref{orbit} provides complete description of all possible cases.

\section{CR geometries of hypersurface type}

\subsection{CR geometries}

Let $M$ be a smooth manifold of dimension $2n+1$ for $n>1$ together with a distribution $\mathcal{H} \subset TM$ of dimension $2n$ and a \emph{complex structure} $J$ on $\mathcal{H}$, i.e., $J:\mathcal{H} \to \mathcal{H}$ is an endomorphism with the property that $J^2=-\id$. The triple $(M,\mathcal{H},J)$ is called a \emph{CR geometry of hypersurface type} if the $i$--eigenspace $\mathcal{H}^{1,0}$ of $J$ in the complexification of $\mathcal{H}$ is integrable, i.e., $[\mathcal{H}^{1,0},\mathcal{H}^{1,0}] \subset \mathcal{H}^{1,0}$.  CR geometry $(M,\mathcal{H},J)$ is called \emph{non--degenerate} if $\mathcal{H}$ is completely non--integrable.

On $\mathcal{H}$ there exists a symmetric bilinear form $h$  with values in the line bundle $TM/\mathcal{H}$ given by $h(\xi,\eta)= \frac12 \pi([\xi,J \eta])$ for all $\xi,\eta \in \Gamma(\mathcal{H})$, where $\pi: TM\to TM/\mathcal{H}$ is a natural projection. Let us recall that $h$ is the real part of the Levi form $\tilde h$ of $(M,\mathcal{H},J)$, while the imaginary part of the Levi form is the map given by $\frac12 \pi([\xi,\eta])$. We assume that $M$ is orientable and denote by $(p,q)$ the signature of the Levi form, where our convention is $p\leq q$, $p+q=n$. Then the signature of $h$ is $(2p,2q)$.

The homogeneous space $PSU(p+1,q+1)/P$ is usually called the \emph{standard model} of non--degenerate CR geometry of hypersurface type of signature $(p,q)$, where the group $PSU(p+1,q+1)$ is the projectivization of the group of matrices preserving the pseudo--Hermitian form 
$$m((u_0,\dots,u_{n+1}), (v_0,\dots,v_{n+1}))=u_0\overline{v_{n+1}}+u_{n+1} \overline{v_0}+\sum_{k=1}^p u_k\overline{v_k}-\sum_{k=p+1}^n u_k\overline{v_k}$$
on $\C^{n+2}$ and $P$ is the stabilizer of the complex line generated by the first basis vector in the standard basis of $\C^{n+2}$. The standard model $PSU(p+1,q+1)/P$ is a smooth real hypersurface in $\C P^{n+1}$ that can be also viewed as the projectivization of the null cone of $m$ in $\C^{n+2}$.

In the rest of the paper, by a CR geometry we mean a non--degenerate CR geometry of hypersurface type of signature $(p,q)$ for $p\leq q$. Such CR geometries can be equivalently described as parabolic geometries modeled on standard models $PSU(p+1,q+1)/P$. This description can be found in \cite[Section 4.2.4]{parabook}. We only use several consequences of this description later in the text.

 \subsection{Distinguished connections}
 
There exist many admissible connections, i.e., connections preserving $\mathcal{H}$ and $J$, on CR geometries. 
In particular, there are several distinguished classes of admissible connections given by a particular normalization condition on the torsion of admissible connections in the class. The most common class is the class of Webster--Tanaka connections \cite [Section 5.2.12]{parabook}. Another important class is the class of Weyl connections \cite[Sections 5.1.2 and 5.2.13]{parabook}. In this paper, we consider the class of Weyl connections, because in our proofs we use  relations between CR transformations and geodesic transformations of normal Weyl connections \cite[Section 5.1.12]{parabook}. 

In fact, Webster--Tanaka connections and Weyl connections induce the same class of distinguished partial connections $\na$ on $\mathcal{H}$. 
Such two distinguished partial connections $\na$ and $\hat \na$ are related by the formula
\begin{align} \label{rozdil}
\begin{split}
\hat \na_{\xi}(\eta)&=\na_{\xi}(\eta)+
F(\xi)\eta+F(\eta)\xi-\tilde h(\xi, \eta)\tilde h^{-1}(F), \\
\end{split}
\end{align} 
where $\xi,\eta$ are vector fields on $\mathcal{H}$ and $F$ is  a one--form in $\mathcal{H}^*$. 
Here $\tilde h^{-1}$ is the inverse of the Levi form $\tilde h$. We will write shortly $\hat \na=\na + F$ instead of the entire formula $(\ref{rozdil})$.

Each Weyl connection $D$ is associated with particular decompositions $TM \simeq \mathcal{H} \oplus \ell$ and $T^*M\simeq \mathcal{H}^* \oplus \ell^*$ that are preserved by $D$, where $\ell$ is a one--dimensional distribution complementary to $\mathcal{H}$. In fact,  one-form $F$ in $\mathcal{H}^*$ from the formula (\ref{rozdil}) also describes the change of the decompositions of $TM$ and $T^*M$ associated with $D$ and $\hat D$. The precise formula for the change of the decompositions can be easily computed using \cite[Section 5.1.5]{parabook}. In general, arbitrary two Weyl connections $D$ and $\hat D$ are related by a suitable
 action of a one--form $\U=\U_1+\U_2$ in $T^*M=\mathcal{H}^* \oplus \ell^*$, where we consider the decomposition associated with the Weyl connection $D$. We write shortly $\hat D=D+ \U_1+\U_2$ instead of the explicit formula for the change, which is  complicated and can be computed using \cite[Section 5.1.6]{parabook}. Let us emphasize that $\U_1$ coincides with $F$ from the formula (\ref{rozdil}) for the corresponding partial connections $\na$, $\hat \na$ determined by $D$, $\hat D$.
 
Let us finally point out that admissible connections provide the fundamental invariant $W$ of  CR geometries which is known as \emph{Chern--Moser tensor} or \emph{Weyl tensor} and coincides with the totally trace--free part of the curvature of arbitrary Weyl or Webster--Tanaka connection. Vanishing of this invariant implies that  CR geometry is \emph{flat}, meaning that  CR geometry is locally equivalent to the standard model $PSU(p+1,q+1)/P$.

\section{Symmetries of CR geometries}
\subsection{Definition of symmetries}\label{sec31}
A \emph{CR transformation} of  CR geometry $(M,\mathcal{H},J)$ is a diffeomorphism of $M$ such that the tangent map preserves the \emph{CR distribution} $\mathcal{H}$ and its restriction to $\mathcal{H}$ is complex linear. We study the following CR transformations.

\begin{def*}\label{def1}
A \emph{symmetry} at $x \in M$ on a CR manifold $(M,\mathcal{H},J)$ is a CR transformation $S_x$ of $M$ such that:
\begin{enumerate}
\item $S_x(x)=x$,
\item $T_xS_x=-$id on $\mathcal{H}$.
\end{enumerate}
We say that  CR geometry is \emph{symmetric} if there exists a symmetry at each point $x \in M$. A \emph{system of  symmetries} on $M$ is a choice of a symmetry $S_x$ at each $x \in M$. We call the system \emph{smooth}, if the map  $S: M \times M \to M$ given by $S(x,y)=S_x(y)$  is smooth in both variables.
\end{def*}

Let us show that the standard model $PSU(p+1,q+1)/P$ is symmetric. The Lie group $PSU(p+1,q+1)$ is the group of all CR transformations of the standard model $PSU(p+1,q+1)/P$, where we consider left action. 
Direct computation gives that all symmetries of the standard model $PSU(p+1,q+1)/P$ at the origin $eP$ are represented by $(1,n,1)$--block matrices of the form
\begin{align}\label{symform}
s_{Z,z}=\left(
\begin{matrix}
-1&-Z& iz+\frac12ZIZ^* \\ 0&E&-IZ^* \\ 0&0&-1
\end{matrix}
\right)
,\end{align}
where $Z \in \C^{n*}$, $z \in \R^*$ are arbitrary, $E$ is the identity matrix of the rank $n$ and $I$ is the diagonal matrix with the first $p$ entries equal to $1$ and the remaining $q$ entries equal to $-1$. 

\begin{lemma*} \label{standard}
There exists an infinite number of symmetries at each point $kP$ of $PSU(p+1,q+1)/P$ given by matrices of the form $ks_{Z,z}k^{-1}$ for all $Z \in \C^{n*}$ and $z \in \R^*$. In particular:
\begin{enumerate}
\item There exists an infinite number of involutive symmetries at each point characterized by the condition $z=0$. For each such symmetry, there is a different metric preserved by this symmetry compatible with the CR geometry.
\item There exists an infinite number of non--involutive symmetries at each point characterized by the condition $z\neq 0$. They do not preserve any metric compatible with the CR geometry. 
\end{enumerate}
\end{lemma*}
\begin{proof}
Each element $s_{Z,z}$ satisfying $z=0$ is conjugated to an element of maximal compact subgroup of $PSU(p+1,q+1)$, and thus preserves the corresponding metric. Each element $s_{Z,z}$ satisfying $z\neq 0$ is contained in a different orbit with respect to the conjugation than the maximal compact subgroup of $PSU(p+1,q+1)$, and thus it cannot preserve any compatible metric.
\end{proof}

The standard model $PSU(p+1,q+1)/P$ is endowed with a pseudo--Riemannian metric compatible with the CR geometry, because the maximal compact subgroup of $PSU(p+1,q+1)$ acts transitively on the standard model.  Moreover, there is exactly one involutive symmetry at each point of this model that is contained in the maximal compact subgroup. These symmetries preserve the corresponding pseudo--Riemannian metric and form a smooth system. This means that in the Riemannian signature, the standard model $PSU(1,n+1)/P$ is symmetric in the sense of \cite{KZ}.

On flat CR geometries, the set of symmetries is locally identical with the one on the standard model. However, these symmetries may not be defined globally. This means that on flat CR geometries, there locally always exists a pseudo--Riemannian metric compatible with the CR geometry preserved by some symmetry at each point. We show in  Example \ref{exam1} that such  pseudo--Riemannian metric compatible with  CR geometry does not have to exist globally.

\subsection{Involutive and non--involutive symmetries}

Suppose that there is a symmetry $S_x$ at $x$ on  CR geometry $(M,\mathcal{H},J)$. If $D$ is a Weyl connection, then $S_x^*D$ is a Weyl connection, too. Therefore, there is a one--form $\U_1+\U_2 \in \mathcal{H}^*\oplus \ell^*$ such that
\begin{align} \label{zmena}
S_x^*D=D+ \U_1+\U_2.
\end{align}

\begin{lemma*} \label{dve-symetrie}
Suppose $S_x$ is a symmetry at $x \in M$. Let $D$ be an arbitrary Weyl connection and let $\U_1+\U_2 \in \mathcal{H}^*\oplus \ell^*$ be the one--form from the formula $(\ref{zmena})$. Then the following claims are  equivalent:
\begin{enumerate}
\item the symmetry $S_x$ is involutive,
\item  $\U_2(x)=0$, and
\item the diffeomorphism $S_x$ is linear in the normal coordinates given by the normal Weyl connection $\bar D$ at $x$ that is uniquely determined by the property that $\bar D$ coincides with the Weyl connection $D+{1 \over 2} \U_1$ at $x$.
\end{enumerate}
Moreover, the partial connection $\nabla^{S_x}$ induced by the Weyl connection $D^{S_x}:=D+{1 \over 2} \U_1$ does not depend on the choice of $D$ at $x$ and satisfies 
\begin{itemize}
\item $S_x^*(\nabla^{S_x})=\nabla^{S_x}$ at $x$, and 
\item $\na^{S_x} W(x)=0$.
\end{itemize}
\end{lemma*}
\begin{proof}
Iterating the formula $(\ref{zmena})$ we compute $$S_x^*S_x^*D=D+ \U_1+ S_x^*(\U_1)+\U_2+S_x^*(\U_2).$$ The component of the (dual) action of $T_xS_x$ on $T_x^*M$
 preserving the decomposition $T_x^*M=\mathcal{H}^*(x)\oplus \ell^*(x)$ is $-\id\oplus \id$, and the component that maps $\mathcal{H}^*(x)$ into $\ell^*(x)$ depends linearly on $\U_1$ and is antisymmetric as a map $\mathcal{H}^*(x)\otimes \mathcal{H}^*(x)\to \ell^*(x)$. Therefore, $ S_x^*(\U_1)(x)=-\U_1(x)$  and $S_x^*(\U_2)(x)=\U_2(x)$.

If the symmetry $S_x$ is involutive, i.e., $S_x^2=\id$, then $$0=\U_2(x)+S_x^*(\U_2)(x)=2\U_2(x)$$ and thus $\U_2(x)=0$. 

If $\U_2(x)=0$, then the normal Weyl connection $\bar D$ that coincides with the Weyl connection $D+{1 \over 2}\U_1$ at $x$ satisfies $$S_x^*(D+{1 \over 2}\U_1)=D+\U_1+S_x^*({1 \over 2}\U_1).$$  At the point $x$, we get $$\U_1(x)+S_x^*({1 \over 2}\U_1)(x)=\U_1(x)-{1 \over 2}\U_1(x)={1 \over 2}\U_1(x)$$ and thus $S_x^*\bar D=\bar D$ follows from the normality \cite[Section 5.1.12]{parabook}. Thus $S_x$ is an affine map, which is linear in the normal coordinates.

If the symmetry $S_x$ at $x$ is linear in the normal coordinates of a Weyl connection, then its (dual) tangent action preserves the decomposition $T_x^*M=\mathcal{H}^*(x)\oplus \ell^*(x)$ and therefore $(T_xS_x)^2=\id$. Then it follows from the linearity that $S_x$ is involutive.

Finally, the last claim follows, because $(\na^{S_x} W)(x)$ is a tensor of type $\otimes^4 \mathcal{H}_x^* \otimes \mathcal{H}_x$ invariant with respect to $S_x$.
\end{proof}

\begin{lemma*} \label{involutivity}
Suppose that there is a symmetry $S_x$ at $x \in M$. 
Let $D$ be an arbitrary Weyl connection and let $\U_1+\U_2 \in \mathcal{H}^* \oplus \ell^*$ be the one--form from the formula $(\ref{zmena})$.
 If $W(x)\neq 0$, then $\U_2(x)=0$ and the symmetry $S_x$ is involutive.
\end{lemma*}
\begin{proof}
Consider the covariant derivative of $W$ with respect to $D+{1 \over 2}\U_1$ in the direction $\ell$ and compute $S_x^*(D+{1 \over 2}\U_1)_{r}W(x)$ for $r\in \ell(x)$. We know that $W(x)$ is $S_x$--invariant and thus 
$$S_x^*(D+{1 \over 2}\U_1)_{r}W(x)=(D+{1 \over 2}\U_1)_{S_x^*(r)}W(x)=(D+{1 \over 2}\U_1)_{r}W(x).$$ 

On the other hand, it generally holds that $S_x^*(D+{1 \over 2}\U_1)=D+{1 \over 2}\U_1+\U_2$ and $\U_2(x)=a\theta(x)$ for a
covector $\theta \in \ell^*(x)$ such that $\theta(r)=1$. Then $$S_x^*(D+{1 \over 2}\U_1)_{r}W(x)=(D+{1 \over 2}\U_1)_{r}W(x)+2aW(x)$$ and thus $a=0$ which implies $\U_2(x)=0$.
\end{proof}

\subsection{Smooth systems of involutive symmetries}

Let us show that the assumption on $W$ to be nowhere vanishing not only implies that all symmetries are involutive, but  also that there is at most one symmetry at each point of $M$ and that these symmetries change smoothly along $M$.

\begin{prop*} \label{hladky-sys}
Suppose that $(M,\mathcal{H},J)$ is a symmetric CR geometry such that $W(x)\neq 0$ for all $x\in M$. Then
\begin{enumerate}
\item there is a unique symmetry $S_x$ at each $x\in M$, 
\item the map $S: x\mapsto S_x$ is smooth, and 
\item $S_x \circ S_y \circ S_x=S_{S_x(y)}$ holds for all $x,y\in M$.
\end{enumerate}

In particular, $(M,S)$ is a reflexion space, i.e., $S: M\times M\to M$ is a smooth map that for all $x,y,z\in M$ satisfies that
\begin{itemize}
\item $S(x,x)=x,$
\item $S(x,S(x,y))=y,$ and
\item $S(x,S(y,z))=S(S(x,y),S(x,z))$.
\end{itemize}

\end{prop*}
\begin{proof}
We show that if there are two different symmetries at $x$ on  CR geometry $(M,\mathcal{H},J)$, then $W$ vanishes at $x$.
Consider two different symmetries $S_x$ and $S'_x$ at $x$ (both must be involutive).
We know from  Lemma \ref{dve-symetrie} that $\na^{S_x} W(x)=0$ and $\na^{S_x'} W(x)=0$ hold for partial connections $\na^{S_x},\na^{S_x'}$. These partial connections are different (at $x$) due to the claim (3) of  Lemma \ref{dve-symetrie}, i.e., $\na^{S_x'}=\na^{S_x}+F$ holds according to the formula (\ref{rozdil}) for $F(x)\neq 0$. This means that the linear map $\mathcal{H}_x\to \mathcal{H}_x$ given by
\begin{align} \label{alg}
\eta\mapsto (F(\xi)\eta+F(\eta)\xi-\tilde h(\xi, \eta)\tilde h^{-1}(F))(x)
\end{align}
defines a non--zero element $\xi(F)(x)$ of a Lie algebra $\frak{csu}(p,q)$ for each $\xi \in \mathcal{H}_x$, where we identify $\frak{csu}(p,q)$ with $$\{X\in \frak{gl}(\mathcal{H}_x) : [X,J_x]=0, h_x(X(\xi),\nu)+h_x(\xi,X(\nu))=a\cdot h_x(\xi,\nu), a\in \mathbb{R}\}.$$
Moreover, the element $\xi(F)(x)$ of $\frak{csu}(p,q)$ has to act trivially on $W(x)$ for all vectors $\xi$.  Let us denote by $\Ann(W_x)$ the set of all $A\in \frak{csu}(p,q)$ such that $A$ acts trivially on $W(x)$. Then we get
$$F(x)\in \Ann(W_x)^{(1)}:=\{F : \xi(F)(x)\in \Ann(W_x)\ {\rm for\ all} \ \xi\in \mathcal{H}_x\}.$$ 
The result of \cite{KT} states that if $W(x)$ is non--trivial, then $\Ann(W_x)^{(1)}=0,$ and thus $\xi(F)(x)=0$ for all $\xi \in \mathcal{H}_x$. Since $\xi(-)(x): \mathcal{H}_x^*\to \frak{csu}(p,q)$ is a linear map at each $x \in M$, this implies $F(x)=0$, which is a contradiction. This proves the uniqueness of symmetries at $x$ in the case $W(x) \neq 0$.

Since $S_x \circ S_y \circ S_x$ is a symmetry at $S_x(y)$, the condition  $S_x \circ S_y \circ S_x=S_{S_x(y)}$  trivially follows from the uniqueness of symmetries. Thus it remains to prove the smoothness of $S$.

Let us fix a partial Weyl connection $\nabla$. For each $y\in M$, there is $F(y)$ such that $(\nabla^{S_y}-\nabla)(y)=F(y)$ by the formula (\ref{rozdil}), which is well--defined due to the uniqueness of $\nabla^{S_y}$ at $y$. Thus $\nabla W(y)$ is given by the algebraic action (\ref{alg}) of $\xi(F(y))$ on $W(y)$ for each $\xi\in \mathcal{H}_y$. Since $\nabla W(y)$ is smooth, the image of $\xi(F(y))$ in $\frak{csu}(p,q)$ depends smoothly on $y$ for each $\xi\in \mathcal{H}_y$. Since the kernel of the action coincides with $\Ann(W_y)^{(1)}$, we conclude that $F(y)$ depends smoothly on $y$.

Let $D$ be an arbitrary Weyl connection inducing the partial Weyl connection $\nabla$. Then $S_y$ is linear in the normal coordinates of the normal Weyl connection $\bar D$ constructed for $D+\frac12F(y)$ due to the claim (3) of  Lemma \ref{dve-symetrie}. Since $\bar D$ depends smoothly on $y$, we get that $S$ is smooth.

It clearly holds that $S_x(x)=x$ and $ S_x^2=\id$ for all $x \in M$. We have proved that $S$ is smooth and satisfies $S_x \circ S_y=S_{S_x(y)} \circ S_x^{-1}=S_{S_x(y)} \circ S_x$ for all $x,y \in M$. Thus it follows that $(M,S)$ satisfies the conditions of the reflexion space.
\end{proof}

 Proposition \ref{hladky-sys} has the following consequence.

\begin{prop*}\label{thm1}
Suppose that $(M,\mathcal{H},J)$ is a symmetric CR geometry. Then either
\begin{enumerate}
\item $W=0$ and the CR geometry is locally equivalent to the standard model, or 
\item $W\neq 0$ and the group generated by symmetries is a Lie group that acts transitively on $M$, i.e.,  CR geometry is homogeneous. In particular, the reflexion space $(M,S)$ from the Proposition \ref{hladky-sys} is a homogeneous reflexion space.
\end{enumerate}
\end{prop*}
\begin{proof}
Suppose that $U \subset M$ consists of all points with non--trivial $W$. It suffices to prove that the group generated by symmetries at points in $U$ acts transitively on $U$ to obtain the claim of the Theorem, because then $W$ is constant on $U$ due to the homogeneity. The fact that the group generated by symmetries on a reflexion space is a Lie group can be found in \cite{L1}.

 Let $c(t)$ be a curve in $U$ such that $c(0)=x$ and ${d \over dt}|_{t=0}c(t)=X\in \mathcal{H}_x$. Then ${d \over dt}|_{t=0}S_{c(t)}(x)$ is tangent to the orbit of the action of the group generated by symmetries at points in $U$. Differentiation of the equality $c(t)=S_{c(t)}c(t)$ gives 
$$X={d \over dt}|_{t=0}S_{c(t)}(c(t))={d \over dt}|_{t=0}S_{c(t)}(x)+T_xS_x.X,$$ and we get 
$${d \over dt}|_{t=0}S_{c(t)}(x)=X-T_xS_x.X=2X.$$ 
Thus at all $x\in U$, the CR distribution $\mathcal{H}$ is tangent to the orbit of the group generated by symmetries at points in $U$. Therefore the group generated by symmetries at points in $U$ acts transitively on $U$.
\end{proof}

Flat symmetric CR geometries do not have to be homogeneous. We construct an explicit example in  Section \ref{sec6}.

\section{Non--flat symmetric CR geometries}

\subsection{Homogeneous CR geometries and their symmetries}

There are several possible ways how to describe a homogeneous CR geometry. We will use the description from  \cite[Section 1.5.15]{parabook} that is closely tight with the setting of Cartan geometries, but as we show in this section, it can be treated independently of the general theory. We need only to recall that the Lie algebra $\frak{su}(p+1,q+1)$ of $PSU(p+1,q+1)$ consists of the $(1,n,1)$--block matrices
$$\left(
\begin{smallmatrix}
a&Z&iz \\ X&A&-IZ^* \\ ix&-X^*I&-\bar a
\end{smallmatrix}
\right),$$
where $\mathfrak{csu}(p,q)=\{(a,A):a \in \C,\ A \in \frak{u}(n),\ a+tr(A)-\bar a=0 \}$, $X \in \C^n $,  $Z \in \C^{n*} $, $x \in \R $ and $z \in \R^*$. This means that we have the following decomposition
$$\frak{su}(p+1,q+1)=\R \oplus \C^{n}\oplus\mathfrak{csu}(p,q)\oplus \C^{n*}\oplus\R^*.$$
The Lie algebra $\fp$ of $P$ corresponds to $(1,n,1)$--block upper triangular part and decomposes as $\fp=\mathfrak{csu}(p,q)\oplus \C^{n*}\oplus\R^*.$ In fact, $P\cong CSU(p,q)\exp(\C^{n*}\oplus\R^*)$, where $CSU(p,q)$ consists of all elements of $P$ preserving the above decomposition.

\begin{lemma*}\label{lem4}
Let $K$ be an arbitrary transitive Lie group of CR transformations of a homogeneous CR geometry $(M,\mathcal{H},J)$ and let $L \subset K$ be the stabilizer of a point. Then there is a pair of maps $(\alpha,i)$ such that $i$ is an injective Lie group homomorphism $i:L \rightarrow P$ and $\alpha$ is a linear map $\alpha: \mathfrak{k}\to \mathfrak{\frak{su}}(p+1,q+1)$ satisfying the following conditions:
\begin{enumerate}
\item $\alpha: \fk \to \mathfrak{su}(p+1,q+1)$ is a linear map extending $T_ei:\mathfrak{l}\to \fp$,
\item $\alpha$ induces an isomorphism $\underline \alpha: \mathfrak{k}/\mathfrak{l} \rightarrow \mathfrak{su}(p+1,q+1)/\fp$ of vector spaces, 
\item $\Ad(i(l))\circ \alpha=\alpha \circ \Ad(l)$ holds for all $l\in L$,
\item the linear map $\wedge^2 \fk\to \frak{su}(p+1,q+1)$ given by the formula $[\alpha(X),\alpha(Y)]-\alpha([X,Y])$ for all $X,Y\in \fk$ has values in $\fp$ and defines a $K$--invariant two--form $\kappa$ with values in $K\times_{\Ad\circ i} \fp$,
\item the component of $\kappa$ in $K\times_{\underline{\Ad}\circ i} \mathfrak{csu}(p,q)$ is a tensor that coincides with $W$, where $\underline{\Ad}$ is the induced action of $P$ on $\mathfrak{csu}(p,q)\cong \fp/(\C^{n*}\oplus\R^*).$
\end{enumerate}

Conversely, suppose that $(\alpha,i)$ is such  pair of maps from $(K,L)$ to $(PSU(p+1,q+1),P)$. Then there is $K$--homogeneous CR geometry $(K/L,\mathcal{H},J)$ satisfying $\mathcal{H}_{eL}=\alpha^{-1}(\C^{n}\oplus \fp)/\fl$ and $J_{eL}=\underline{\alpha}^*(J)$, where $J$ is the complex structure on $\C^n$.
\end{lemma*}
A pair $(\alpha,i)$ satisfying the conditions (1)--(3) of  Lemma \ref{lem4} is usually called an \emph{extension} of $(K,L)$ to $(PSU(p+1,q+1),P)$. The two--form $\kappa$ from the condition (4) is the curvature of the Cartan connection given by the extension $(\alpha,i)$. Finally, the condition (5) is the normalization condition on the curvature $\kappa$ that can be also expressed as $\partial^*\kappa=0$, where $\partial^*$ is the Kostant's co--differential \cite[Section 3.1.11]{parabook}.
\begin{proof}
It is shown in \cite[Section 1.5.15]{parabook} that each homogeneous Cartan (and thus parabolic) geometry can be described by a particular extension and that each extension determines a homogeneous Cartan geometry. The formula for $\kappa$ in the condition (4) is obtained from \cite[Section 1.5.16]{parabook}. Therefore, it follows from the description of CR geometries in \cite[Section 4.2.4]{parabook} that the conditions  (4) and (5) on the curvature $\kappa$ have to be satisfied.
\end{proof}

\begin{def*}
The pair $(\alpha,i)$ from  Lemma \ref{lem4} is called the \emph{normal extension} of $(K,L)$ to $(PSU(p+1,q+1),P)$ describing the homogeneous CR geometry $(M,\mathcal{H},J)$. 
\end{def*}

Examples of normal extensions describing certain homogeneous CR geometries and the explicit formula from the condition (5) of  Lemma \ref{lem4} can be found in \cite{HG-dga}.

It is clear from the second part of  Lemma \ref{lem4} that only the maps $i$ and $\underline{\alpha}$ are sufficient to determine  CR geometry. This means that there are many normal extensions $(\alpha,i)$ of $(K,L)$ to $(PSU(p+1,q+1),P)$ describing the same CR geometry.  The other parts of $\alpha$ are completely determined by the condition (5) from  Lemma \ref{lem4} and carry the information about Weyl connections. The remaining freedom (for fixed $i$) is in the choice of a complex basis of $\underline{\alpha}^{-1}(\C^n)$. In general, if $h\in P$, then the pair $(\Ad(h)\circ \alpha,conj(h)\circ i)$ is also a normal extension of $(K,L)$ to $(PSU(p+1,q+1),P)$ describing the same CR geometry as the normal extension $(\alpha,i)$.

Let us summarize the results characterizing symmetric non--flat homogeneous CR geometries following from \cite{GZ-Lie,my-dga}.

\begin{prop*}\label{existence-2}
Let $K$ be the Lie group of all CR transformations of a non--flat homogeneous CR geometry $(M,\mathcal{H},J)$. Then the following is equivalent:

\begin{enumerate}
\item There is a (unique) symmetry at each point.
\item There is $s\in L$ such that the triple $(K,L,s)$ is a (non--prime) homogeneous reflexion space, i.e., 
\begin{itemize}
\item $s$ commutes with all elements of $L$, 
\item $s^2=e$, where $e$ is the identity element of $L$, and 
\item all symmetries are of the form $S_{kL}=ksk^{-1}$ for $k\in K$.
\end{itemize}
\item There is a normal extension $(\alpha,i)$ of $(K,L)$ to $(PSU(p+1,q+1),P)$ describing $(M,\mathcal{H},J)$ such that $i(L)\subset CSU(p,q)$ and $s_{0,0}\in i(L)$ (see the formula (\ref{symform})).
\item For each normal extension $(\alpha,i)$ of $(K,L)$ to $(PSU(p+1,q+1),P)$ describing $(M,\mathcal{H},J)$, there is a (unique) $Z\in \C^{n*}$ such that $\Ad(\exp (Z))\alpha(\fk)$ is preserved by $\Ad(s_{0,0})$, and the Lie algebra automorphism of $\fk$ given by $\Ad(s_{0,0})$ defines an automorphism of the Lie group $K$. 
\end{enumerate}
\end{prop*}

The condition (3) of  Proposition \ref{existence-2} immediately implies that there are $K$--invariant Weyl connections on a symmetric non--flat CR geometry $(M,\mathcal{H},J)$. According to \cite[Proposition 1.4.8]{parabook}, a $K$--invariant connection on $T(K/L)$ can be described by a map $\gamma: \fk\to \frak{gl}(\fk/\fl)$ such that 
\begin{itemize}
\item $\gamma|_\fl=\underline{\ad}$, and 
\item $\gamma(\Ad(h)(X))=\underline{\Ad}(h)\circ \gamma(X)\circ \underline{\Ad}(h)^{-1}$
\end{itemize}
hold for all $X\in \fk$ and $h\in L$, where $\underline{\Ad}: L\to Gl(\fk/\fl)$ is induced by the adjoint representation.

\begin{prop*}\label{connexe}
Let $K$ be the Lie group of all CR transformations of a non--flat symmetric CR geometry $(M,\mathcal{H},J)$. Let $(\alpha,i)$ be a normal extension of $(K,L)$ to $(PSU(p+1,q+1),P)$ describing $(M,\mathcal{H},J)$ such that $i(L)\subset CSU(p,q)$ and $s_{0,0}\in i(L)$. Then $\gamma:=\underline{\alpha}^*(\underline{\ad}\circ r_0)$ describes a $K$--invariant Weyl connection, where $r_0: \mathfrak{su}(p+1,q+1)\to \mathfrak{csu}(p,q)$ is the projection along $\R \oplus \C^{n}\oplus \C^{n*}\oplus\R^*$.

In particular, there is a bijection between the set of $K$--invariant Weyl connections on $M$ and the set of $z\in \R^*$ such that $conj(\exp(z))\circ i(L)\subset CSU(p,q)$ holds for the extension $(\alpha,i)$.
\end{prop*}
\begin{proof}
We proved the existence of $K$--invariant Weyl connections on non--flat symmetric CR geometries in \cite{my-dga}. Therefore it is enough to check that they can be described by the functions $\gamma$. Since $i(L)\subset CSU(p,q)$, the projection $r_0$ is $i(L)$--equivariant and $\gamma|_\fl=\underline{\ad}$ holds. Therefore each  $\gamma$ describes a $K$--invariant connection. The fact that this is a Weyl connection follows directly from the condition (5) in  Lemma \ref{lem4}.

It is clear that the one--form $\U_1+\U_2$ measuring the `difference' between two $K$--invariant Weyl connections is given by an $i(L)$--invariant element of $\C^{n*} \oplus \R^*$. Since $s_{0,0}\in i(L)$, it has to be an element of $\R^*$. It is clear that $z\in \R^*$ is $i(L)$--invariant element if and only if $conj(\exp(z))\circ i(L)\subset CSU(p,q)$ holds.
\end{proof}

\subsection{
Groups generated by symmetries}\label{sec4.2}
The following Theorem significantly improves the characterization of non--flat symmetric homogeneous CR geometries given by Propositions \ref{thm1} and \ref{existence-2}.

\begin{thm*} \label{thm2}
Let $K$ be the Lie group generated by all symmetries of a non--flat symmetric CR geometry $(M,\mathcal{H},J)$. Let $(\alpha,i)$ be a normal extension of $(K,L)$ to $(PSU(p+1,q+1),P)$ describing the CR geometry that satisfies $i(s)=s_{0,0}$ and $i(L)\subset CSU(p,q)$. Denote by $\fh$ the $1$--eigenspace of $s$ in $\fk$ and by $\fm$ the $-1$--eigenspace of $s$ in $\fk$. Then:
\begin{enumerate}
\item The following conditions hold
\begin{itemize}
\item $\alpha(\fl)\subset \frak{u}(p,q)$, 
\item $\alpha(\fm)\subset \C^{n}\oplus \C^{n*}$, and 
\item $\alpha(\fh)\subset \R \oplus\mathfrak{csu}(p,q)\oplus \R^*$ is a Lie subalgebra.
\end{itemize}
\item There is a basis of $\fh/\fl\oplus \fm$ such that for a vector in $\fh/\fl\oplus \fm$ with coordinates $(x,X)$ holds
$$\alpha((x,X)+\fl)=\Ad(\exp(z))\circ \left(
\begin{matrix}
aix&\Rho_1(X)& \Rho_2(x)i \\ X&Ax&-I\Rho_1(X)^* \\ xi&-IX^*&aix
\end{matrix}
\right)+\alpha(\fl),$$
where $z\in \R^*$, $\Rho_1:\C^{n}\to \C^{n*}$, $\Rho_2: \R\to \R^*$
and $(a,A)\in \frak{u}(p,q)$ normalizes $\alpha(\fl)$. 
\item The maps $\Rho_1,\Rho_2$ and the matrix $(a,A)$ are completely determined by the condition (5) from  Lemma \ref{lem4}.
\end{enumerate}
\end{thm*}
\begin{proof}
We know from  Proposition \ref{existence-2} that there exists a normal extension $(\alpha,i)$ of $(K,L)$ to $(PSU(p+1,q+1),P)$ satisfying our assumptions. 

Consider the canonical decomposition $\fk = \fh \oplus \fm$, where  $\fh$ is $1$--eigenspace of $s$ and $\fm$ is $-1$--eigenspace of $s$. Then $\alpha(\fm)\subset \C^{n}\oplus \C^{n*}$ and $\alpha(\fh)\subset \R \oplus\mathfrak{csu}(p,q)\oplus \R^*$ follow from the assumption $i(s)=s_{0,0}$ and $\alpha(\fh)$ is a Lie subalgebra, because $dim(\fh/\fl)=1$. We can identify $\fm$ with $\C^{n}$ via $\alpha$, because the restriction of $\alpha$ to the map $\fm\to \C^{n}$ is injective. Indeed, if the restriction is not injective, then the elements in its kernel would be another symmetries at $eL$, but we know that there is only one symmetry. This identification uniquely determines the map $i: L\to CSU(p,q)$. 

Further, $[\fm,\fm] \subset \fh$ holds and we have the corresponding symmetric space $K/H^0$, where $H^0$ is the connected component of identity of the fixed point set of the conjugation by $s$. Therefore $\exp([X,Y]) \in H^0$ for each $X,Y \in \fm$. The map $\Ad: H^0\to GL(\fm)$ can be restricted to the connected component of identity $L^0$ of $L$ and the restriction coincides with $i$. Therefore, it suffices to show that the element $\ad([X,Y])\in \frak{gl}(\fm)$ belongs to $\frak{sl}(\fm)$ for all $X,Y\in \fm$. But we have 
$\ad([X,Y])=\ad(X) \circ \ad(Y) - \ad(Y) \circ \ad(X)$ and the trace equals to
$$
tr(\ad([X,Y]))=tr(\ad(X) \circ \ad(Y) - \ad(Y) \circ \ad(X))=B(X,Y)-B(Y,X)
,$$
where $B$ denotes the Killing form, which is symmetric. Therefore $i(L^0) \subset U(p,q)$ and $T_ei(\fl)\subset \frak{u}(p,q)$. In particular, the claim (1) holds. The map $\alpha$ can be expressed as in the claim (2), because there is always $z\in \R^*$ such that the extension $(\Ad(\exp(-z))\circ \alpha,conj(\exp(-z))\circ i)$ satisfies
$$\Ad(\exp(-z))\circ \alpha((x,0)+\fl)= \left(
\begin{matrix}
aix&0& \Rho_2(x)i \\ 0&Ax&0 \\ xi&0&aix
\end{matrix}
\right)+\Ad(\exp(-z))\circ \alpha(\fl).$$
Since the CR geometry $(M,\mathcal{H},J)$ does not depend on parts $\Rho_1,\Rho_2$ and $(a,A)$ of $\alpha$, these parts are completely determined by the condition (5) from  Lemma \ref{lem4}.
\end{proof}

Let us remark that although the Lie algebra homomorphism $i$ is uniquely determined by the isomorphism $\fm\cong \C^{n}$ given by $\alpha$, the converse is not true. See \cite{HG-dga} for examples of non--equivalent CR geometries described by extensions with the same Lie group homomorphism $i$.

Let us further remark that we are not aware of any example of an extension $(\alpha,i)$ of $(K,L)$ to $(PSU(p+1,q+1),P)$ where  $z\in \R^*$ from the claim (2) of  Theorem \ref{thm2}  does not correspond to an invariant Weyl connection. The main reason for this is the following result.

\begin{prop*}\label{prop5}
Suppose that $\Ad(L^0)|_{\fh/\fl}=\Ad(L)|_{\fh/\fl}$. Then $i(L)\subset U(p,q)$ and there is a bijection between $ \R^*$ and the set of $K$--invariant Weyl connections. In particular, there is a unique $K$--invariant Weyl connection corresponding to the normal extension $(\alpha,i)$ satisfying $i(L)\subset U(p,q)$, $i(s)=s_{0,0}$ and
\begin{align} \label{form-prop5}
\alpha((x,X)+\fl)=\left(
\begin{matrix}
aix&\Rho_1(X)& \Rho_2(x)i \\ X&Ax&-I\Rho_1(X)^* \\ xi&-IX^*&aix
\end{matrix}
\right)+\alpha(\fl).
\end{align}
This particularly holds when the transitive group $K$ is semisimple.
\end{prop*}
\begin{proof}
If $\Ad(L^0)|_{\fh/\fl}=\Ad(L)|_{\fh/\fl}$, then $i(L)\subset U(p,q)$ holds and the claim follows. It follows from the classification of semisimple symmetric spaces that $H$ is reductive and there is a complement to $\fl$ in the center of $\fh$. Consequently $\Ad(L^0)|_{\fh/\fl}=\Ad(L)|_{\fh/\fl}$.
\end{proof}

\subsection{Relations to CR algebras}

We explain here relations between our concept and the concept of CR algebras introduced in \cite{AMN}. We denote here by $\frak{n}_\C$ the complexification of a Lie algebra $\frak{n}$.

Let $(\alpha,i)$ be an extension of $(K,L)$ to $(PSU(p+1,q+1),P)$. We complexify the linear map $\alpha$ to obtain a map
$$\alpha_{\mathbb{C}}: \fk_{\mathbb{C}}\to \frak{sl}(n+2,\mathbb{C}).$$
The Lie algebra $\frak{sl}(n+2,\mathbb{C})$ decomposes as 
$$\frak{sl}(n+2,\mathbb{C})=\mathbb{C}\oplus (\mathbb{C}^n\oplus \mathbb{C}^{n*})\oplus (\frak{gl}(n,\mathbb{C})\oplus \mathbb{C})\oplus (\mathbb{C}^{n*}\oplus \mathbb{C}^n) \oplus \mathbb{C},$$ where $\fp_{\mathbb{C}}= (\frak{gl}(n,\mathbb{C})\oplus \mathbb{C})\oplus (\mathbb{C}^{n*}\oplus \mathbb{C}^n) \oplus \mathbb{C}$. The subspace $\mathbb{C}^{n*}\oplus \fp_{\mathbb{C}}$ is a parabolic subalgebra of $\frak{sl}(n+2,\mathbb{C})$ that satisfies $$\mathcal{H}^{01}_{eL}=\alpha^{-1}_{\mathbb{C}}(\mathbb{C}^{n*}\oplus \fp_{\mathbb{C}})/\fl_{\mathbb{C}}.$$
Therefore, the preimage $\fq$ of $\mathcal{H}^{01}_{eL}$ in $\fk_{\mathbb{C}}$ is a Lie subalgebra of the form $$\fq=\alpha^{-1}_{\mathbb{C}}(\mathbb{C}^{n*}\oplus \fp_{\mathbb{C}}).$$

The pair $(\fk,\fq)$ satisfies conditions of \emph{CR algebra} from \cite[Section 1.2.]{AMN}.
 It is proved in \cite{AMN} that this is the minimal set of data describing CR geometry on the homogeneous space $K/L$. However,  CR algebras do not provide as much information as the extension $(\alpha,i)$.
 In particular, we cannot obtain directly the curvature $\kappa$ of the corresponding Cartan connection from the CR algebra. Therefore, it is not easy to distinguish whether two CR algebras correspond to equivalent CR geometries.

There are conditions in \cite[Section 1.4]{AMN} that characterize CR algebras of CR geometries that are symmetric in the sense of \cite{KZ}. 
One of these conditions ensures that there is a Riemannian metric compatible with CR geometry. Other conditions are analogous to the condition (4) of Proposition \ref{existence-2} which says that the Lie algebra automorphism of $\fk$ given by $\Ad(s_{0,0})$ defines an automorphism of the Lie group $K$.

There is the following method to check whether CR geometries corresponding to CR algebras $(\fk,\fq)$ are symmetric (in our sense) and to construct the normal extensions $(\alpha,i)$ that describe them.
\\[1mm]
$(1)$ We consider $\fl=\fk \cap \fq\cap \bar{\fq}$ and $\mathcal{H}_{eL}=\fk/\fl \cap (\fq+\bar{\fq})/(\fq\cap \bar{\fq})$, where $\bar{\fq}$ is the subalgebra conjugated  to $\fq$ in $\fk_{\mathbb{C}}$.
\\[1mm]
$(2)$ We choose a complex basis of $\mathcal{H}_{eL}$. This choice defines a Lie algebra homomorphism $\fl\to \frak{csu}(p,q)$ and the following facts hold:
\\$(2a)$
 If this homomorphism is not injective, then CR geometry is flat (we will discuss this situation later).  
\\$(2b)$
 If this homomorphism is injective and the CR geometry is symmetric, then it coincides with the restriction of $\alpha$ to $\fl$ for some normal extension $(\alpha, i)$ describing the CR geometry.
\\$(2c)$
 If this homomorphism is injective and the CR geometry is not symmetric, then the homomorphism corresponds only to associated graded map corresponding to restriction of $\alpha$ to $\fl$ for some normal extension $(\alpha, i)$ describing the CR geometry.
\\[1mm]
$(3)$ Each choice of representatives (in $\fk$) of the complex basis of $\mathcal{H}_{eL}$ from (2) together with a choice of an element of $\fk$ complementary to $\mathcal{H}_{eL}$ allows us to define
\\$(3a)$ a linear map $\alpha$ of the form $(\ref{form-prop5})$ from Proposition \ref{prop5} for (at this point) unknown linear maps $a,A,\Rho_1,\Rho_2$,
\\$(3b)$ a linear map $\tau: \wedge^2 \fk\to \frak{su}(p+1,q+1)$  given for all $X,Y\in \fk$ by the formula $$\tau(X,Y):=[\alpha(X),\alpha(Y)]-\alpha([X,Y]),$$ 
$(3c)$ a linear map $\nu: \fk\to \fk$ such that $\nu$ equals to
\begin{itemize}
\item
  $-\id$ on the  representatives (in $\fk$) of  complex basis of $\mathcal{H}_{eL}$, and
 \item  $\id$ on the element of $\fk$ complementary to $\mathcal{H}_{eL}$ and on $\fl$. 
\end{itemize}
Moreover, we consider only the choices that satisfy the equivalent conditions of the following statement.
\begin{prop*}
The map $\nu$ is a Lie algebra automorphism of $\fk$ if and only if the components $$(\mathbb{R}\oplus \alpha(\fl))\otimes \mathbb{C}^n\to \mathbb{R}\oplus \frak{csu}(p,q)\oplus \mathbb{R}^*, \ \ \  \mathbb{C}^n\otimes \mathbb{C}^n\to \mathbb{C}^{n}\oplus \mathbb{C}^{n*}$$ of $\tau$ vanish for all linear maps $a,A,\Rho_1,\Rho_2$. 
\end{prop*}
\begin{proof}
A consequence of the formula for $\tau$ is that $\nu$ is a Lie algebra automorphism of $\fk$ if and only if $$\Ad(s_{0,0})\tau(\nu(X),\nu(Y))=\tau(X,Y),\ \ \  \Ad(s_{0,0})[\alpha(\nu(X)),\alpha(\nu(X))]=[\alpha(X),\alpha(Y)]$$ hold for all $X,Y\in \fk$. If $\alpha$ is of the form $(\ref{form-prop5})$, then $$\Ad(s_{0,0})[\alpha(\nu(X)),\alpha(\nu(X))]=[\alpha(X),\alpha(Y)]$$ holds for all $X,Y\in \fk$ and all linear maps $a,A,\Rho_1,\Rho_2$, and $$\Ad(s_{0,0})\tau(\nu(X),\nu(Y))=\tau(X,Y)$$ holds for all $X,Y\in \fk$ if and only if the claimed components vanish.
\end{proof}
\noindent
$(4)$  There are the following possibilities for the choice in the step $(3)$.
\\ $(4a)$
If there is no choice such that $\nu$ is a Lie algebra automorphism of $\fk$, then the CR geometry corresponding to CR algebra $(\fk,\fq)$ is not symmetric.
\\$(4b)$ If there is a choice such that $\nu$ is a Lie algebra automorphism of $\fk$, then the CR geometry corresponding to CR algebra $(\fk,\fq)$ is symmetric if and only if $\nu$ induces Lie group automorphism of $K$ and $L$ is contained in fixed point set of $\nu$.
\\[1mm]
$(5)$ 
We require from now that CR geometry corresponding to CR algebra $(\fk,\fq)$ is symmetric. The remaining step is to determine the choice of an element of $\fk$ complementary to $\mathcal{H}_{eL}$ and $i:L\to P$ such that $(\alpha, i)$ is a normal extension describing the CR geometry. 
 We know that there is a choice such that $(\Ad(\exp(z))\circ \alpha,i')$ is an extension for some $z\in \mathbb{R}^*$, where the Lie group homomorphisms $i': L\to P$ is induced by (adjoint) action of $L$ on $\mathcal{H}_{eL}$ and $\alpha(\fl)$. Thus it suffices to check the vanishing of components $$\alpha(\fl)\otimes \mathbb{R}\to \frak{su}(p+1,q+1), \ \ \  \mathbb{C}^n\otimes \mathbb{C}^n\to \mathbb{R}$$ of $\tau$ for all linear maps $a,A,\Rho_1,\Rho_2$.
The condition (5) of Lemma \ref{lem4} provides linear equations that determine uniquely the linear maps $a,A,\Rho_1,\Rho_2$ for which the extension $(\alpha,i)$ is normal.

\subsection{Example of non--flat symmetric CR geometries}

Consider a  Lie group $E(2)=\R^2 \oplus \frak{so}(2)$ of isometries of Euclidean plane. There is the following normal extension $(\alpha,i)$ of $(E(2),\{\id\})$ to $(PSU(1,2),P)$ of the form
 \begin{align} \label{neE2}
&\alpha
\left(   \begin {matrix} 
0&0&0\\ 
{x\over 2}&0&-X_1
\\ X_2 &X_1&0\end{matrix}  \right)= \left( \begin{matrix}  
 {ix \over 16} & -\frac{5}{16}X_1-\frac{3i}{16}X_2 &-{\frac {15ix}{256}}\\
X_1+iX_2&-{ix \over 8}& \frac{5}{16}X_1-\frac{3i}{16}X_2 \\ ix&-X_1+iX_2& {ix \over 16}\end{matrix} \right),
\end{align}
where the choice of the basis of the Lie algebra of $\R^2 \oplus \frak{so}(2)=\langle x,X_2\rangle\oplus \langle X_1\rangle$ reflects the convention from Section \ref{sec4.2}, i.e., $(x,(X_1,X_2))$ are the distinguished coordinates from Theorem \ref{thm2}.
Indeed, since $i$ is trivial and
\begin{align*}
&\tau\left((x,(X_1,X_2)),(y,(Y_1,Y_2))\right)= \\
&\left( \begin{matrix} 
0& \frac {3i}{32} yX_1-\frac {3}{32}yX_2
 -\frac {3i}{32} xY_1+\frac {3}{32} xY_2&0\\ 0&0& \frac {3i}{32} yX_1+\frac {3}{32}yX_2
 -\frac {3i}{32} xY_1-\frac {3}{32} xY_2\\ 
0&0& 0\end{matrix} \right)
\end{align*}
holds for the linear map $\tau$ determining the curvature $\kappa$, 
it follows that $(\alpha,i)$ is a normal extension describing a non--flat symmetric CR geometry.

In fact, any linear invertible linear map 
$B: \mathbb{R}^2\oplus \frak{so}(2)\to \mathbb{R}\oplus \mathbb{C}$ defines a CR algebra $(\fk,\fq)$ for $$\fq=B^{-1}_{\mathbb{C}}(\mathbb{C}^*\oplus \fp_{\mathbb{C}})$$
and we ask the following question: Which maps $B$ correspond to non--equivalent non--degenerate symmetric CR geometries of hypersurface type on the Lie group $E(2)$ of isometries of Euclidean plane?

We give the answer to this question (using the algorithm from previous section and \cite[Lemma 3.5]{HG-dga}) in the following statement.

\begin{prop*} \label{sedm}
The normal extension $(\alpha, i)$ of the form $(\ref{neE2})$ describes the unique (up to equivalence) non--degenerate symmetric CR geometry of hypersurface type on the Lie group $E(2)$.
\end{prop*}
\begin{proof}
Consider an invertible linear map $B: \mathbb{R}^2\oplus \frak{so}(2)\to \mathbb{R}\otimes \mathbb{C}^2$. 
The construction of the objects from the algorithm is clear in this case.
We need to find for which maps $B$ the components 
$$\mathbb{R}\otimes \mathbb{C}\to \mathbb{R}\oplus \frak{csu}(1)\oplus \mathbb{R}^*, \ \ \  \mathbb{C}\otimes \mathbb{C}\to \mathbb{C}\oplus \mathbb{C}^*,\ \ \  \ \mathbb{C}\otimes \mathbb{C}\to \mathbb{R}$$
of $\tau$ vanish for all linear maps $a,A,\Rho_1,\Rho_2$. In fact, this provides three equations on the entries of the matrix $B$ that can be solved explicitly. In the standard bases of $\mathbb{R}^2\oplus \frak{so}(2)$ and $\mathbb{R}\otimes \mathbb{C}$, the inverses of matrices $B$ that satisfy these equations define the following subvariety:
\begin{align} \label{var}
 \left( \begin{matrix} 
{p_1p_2 -p_3p_4 \over 2} & p_5 & {p_5p_3 -2p_6 \over 2}\\
p_6&p_4&p_2\\ 
0&p_1&p_3 
\end{matrix} \right).
\end{align}
Thus it remains to check the action of morphisms from \cite[Lemma 3.5]{HG-dga} that determine which extensions define equivalent CR geometries. In particular, there are 
\begin{itemize}
\item four--dimensional Lie group of derivations of $\mathbb{R}^2\oplus \frak{so}(2)$ that in addition contains the homothethies, and
\item two--dimensional Lie subgroup that forms center of $CSU(p,q)$. 
\end{itemize}
We compute that the induced action of these morphisms on the six--dimensional variety $(\ref{var})$ is transitive and the matrix
$$
\left( \begin{matrix}  {1 \over 2}&0&0\\
0&0&1\\ 
0&1&0 
\end{matrix} \right)
$$
corresponds to the extension $(\ref{neE2})$. 
\end{proof}

\section{Metrizability and CR embeddings}

In this section, we always consider the $K$--invariant Weyl connection $D$ corresponding to a normal extension $(\alpha,i)$ describing a homogeneous CR geometry $(M,\mathcal{H},J)$ that satisfies $i(L)\subset U(p,q)$, $i(s)=s_{0,0}$ and
$$\alpha((x,X)+\fl)=\left(
\begin{matrix}
aix&\Rho_1(X)& \Rho_2(x) \\ X&Ax&-I\Rho_1(X)^* \\ x&-IX^*&aix
\end{matrix}
\right)+\alpha(\fl).$$ Moreover, we always assume $\Ad(L^0)|_{\fh/\fl}=\Ad(L)|_{\fh/\fl}$, where $L^0$ is the component of identity of $L$. This gives almost no restriction, because this condition is always satisfied on the symmetric CR geometry on the covering $K^0/L^0\to K/L$. 

\subsection{Distinguished metrics compatible with the CR geometry}

The symmetric bilinear form $h$ generally does not define a pseudo--Riemannian metric on $\mathcal{H}$, because there is no natural way, how to measure the length of elements of $TM/\mathcal{H}$. The situation is different, if there is a Weyl connection preserving not only the decomposition $\mathcal{H} \oplus \ell$, but also a non--zero vector field $r$ in $\ell$. Such  Weyl connection is called \emph{exact} and the vector field $r$ is called the \emph{Reed field}. Equivalently, each exact Weyl connection corresponds to the \emph{contact form} $\theta$ that annihilates $\mathcal{H}$ and satisfies $\theta(r)=1$ for the Reeb field $r$. 
If there is an exact Weyl connection, then $\theta\circ h$ is a pseudo--Riemannian metric on $\mathcal{H}$. This metric is compatible with the CR--structure, because the form $h$ satisfies $h(J\xi,J\nu)=h(\xi,\nu)$  for all sections $\xi,\nu$ of $\mathcal{H}$. The exact Weyl connection preserves this metric and the Reeb field can be used to construct a pseudo--Riemannian metric on $TM$, for which the connection is a metric connection. This metric is usually called a \emph{Webster metric}. However, the Webster metric neither has to exist nor has to be compatible with the symmetries. Therefore, if we want to find a metric compatible with the CR geometry that is preserved by all symmetries, we need to show that the distinguished Weyl connection $D$ is exact.

\begin{thm*} \label{thm3}
Let $K$ be the Lie group generated by all symmetries of a non--flat symmetric CR geometry $(M,\mathcal{H},J)$. Suppose that $\Ad(L^0)|_{\fh/\fl}=\Ad(L)|_{\fh/\fl}$. The distinguished Weyl connection $D$ is exact and furthermore, there exists 
\begin{itemize}
\item a $K$--invariant contact form $\theta$,
\item a $K$--invariant pseudo--Riemannian metric $\bar g:= \theta\circ h$ on $\mathcal{H}$, and 
\item a $K$--invariant Webster metric $g:=\theta\circ h+\theta\otimes \theta$ on $TM$ 
\end{itemize}
such that
\begin{enumerate}
\item $D\bar g=0, Dg=0$,
\item $g|_{\mathcal{H}}=\bar g$ and the Reeb field of $D$ is orthogonal to $\mathcal{H}$ and has length $1$,
\item choosing the Reeb field of $D$ as a trivialization of $TM/\mathcal{H}\otimes \C$, the pseudo--Riemannian metric $\bar g$ on $\mathcal{H}$ coincides with the real part of the Levi form up to a constant multiple,
\item the symmetry at $x$ is linear in geodesic coordinates of $D$ at $x$, reverses the directions of $\mathcal{H}_x$ and preserves the direction of the Reeb field of $D$ at $x$.
\end{enumerate}
\end{thm*}
\begin{proof}
The image of $\alpha$ is contained in $\R \oplus\C^n \oplus \frak{u}(p,q) \oplus\C^{n*} \oplus \R^*$ and thus $\gamma$ describing the corresponding $K$--invariant Weyl connection has values in $\underline{\ad}(\frak{u}(p,q))$. Furthermore, the assumption $\Ad(L^0)|_{\fh/\fl}=\Ad(L)|_{\fh/\fl}$ implies that $i(L)\subset U(p,q)$ and therefore the maps $\underline{\ad}^{-1}\circ \gamma$ and $i$ satisfy all the conditions of \cite[Theorem 1.4.5]{parabook}. This means that the Weyl connection $D$ is an associated connection to a $K$--invariant principal connection on the bundle $K\times_{i(L)}U(p,q)\to K/L$. Therefore it is an exact Weyl connection, because its holonomy is contained in $U(p,q)$. The remaining claims then follow from  general theory.
\end{proof}

In the Riemannian signature,  Theorem \ref{thm3} particularly allows to compare symmetric CR geometries (in our sense) with the symmetric CR geometries in the sense of \cite{KZ}, because we have found a metric compatible with  CR geometry that is preserved by all symmetries.

\begin{thm*} \label{thm4}
Suppose that $p=0$. Then each non--flat symmetric CR geometry is covered by a symmetric CR geometry in the sense of \cite{KZ}, where the covering is a CR map that intertwines the symmetries.
\end{thm*}

\subsection{CR embeddings}

Consider the fiber bundle $K\times_{i} CSU(p,q)/U(p,q)\to K/L$. If $\Ad(L^0)|_{\fh/\fl}=\Ad(L)|_{\fh/\fl}$ holds, then this bundle is trivial, i.e., 
$$K\times_{i} CSU(p,q)/U(p,q)=K/L\times \mathbb{R}.$$
Let us prove the following statement:

\begin{thm*}\label{embed}
Let $K$ be the Lie group generated by all symmetries of a non--flat symmetric CR geometry $(M,\mathcal{H},J)$. Suppose that $\Ad(L^0)|_{\fh/\fl}=\Ad(L)|_{\fh/\fl}$. Then: 
\begin{enumerate}
\item the manifold $K/L\times \mathbb{R}$ is a complex manifold, and 
\item the inclusion $K/L\to K/L\times \mathbb{R}$ given as a zero section is a CR embedding.
\end{enumerate}
\end{thm*}
\begin{proof}
We need some more details from the theory of Cartan geometries from \cite[Sections 1.5.13 and 3.1.2]{parabook} to proceed with the proof. First, there is a natural complement of $\frak{u}(p,q)$ in $\frak{csu}(p,q)$ given by  so--called grading element, which is the unique element $Z\in \frak{csu}(p,q)$ acting by $-2$ on $\mathbb{R}$, $-1$ on $\C^n$, $0$ on $\frak{csu}(p,q)$, $1$ on $\C^{n*}$ and $2$ on $\R^*$. Furthermore, there is a Cartan connection on $K/L\times \mathbb{R}$ induced by  CR geometry, where we identify $\mathbb{R}$ (via $\exp$) with the multiples of the grading element $Z$. Then the Weyl connection $D$ provides a reduction of this Cartan connection to $U(p,q)$, which allows us to identify the tangent space of $K/L\times \mathbb{R}$ with the fiber bundle $(K\times \mathbb{R})\times_{i} (\R \oplus\C^n\oplus \frak{csu}(p,q)/\frak{u}(p,q))$. We can extend the complex structure on $\C^n$ to $\R \oplus\C^n\oplus \frak{csu}(p,q)/\frak{u}(p,q)$ by declaring $\R$ to be the imaginary part of $\mathbb{C}$ and the multiples of the grading element in $ \frak{csu}(p,q)/\frak{u}(p,q)$ to form the real part of $\mathbb{C}$. This definition is clearly $U(p,q)$--invariant (and thus $K$--invariant) and defines an almost complex structure $J$ on $K\times_{i} CSU(p,q)/U(p,q)$. 

Let us compute the Nijenhuis tensor $[\xi,\eta]-[J\xi,J\eta]+J([J\xi,\eta]+[\xi,J\eta])$ of $J$ for $\xi,\eta\in T(K/L\times \mathbb{R})$. For each $x\in K/L\times \mathbb{R}$, there are vector fields $\tilde \xi, \tilde \eta$ such that $\tilde \xi(x)=\xi(x), \tilde \eta(x)=\eta(x)$ and that the element $[\tilde \xi,\tilde \eta](x)$ is identified with the element $$[X,Y]-[\alpha(X+\fh),\alpha(Y+\fh)]+\alpha([X+\fh,Y+\fh]) \mod  \frak{u}(p,q)\oplus\C^{n*} \oplus \R^*,$$ 
where $\xi(x), \eta(x)$ are identified with $X,Y\in \R \oplus\C^n\oplus \frak{csu}(p,q)/\frak{u}(p,q)$. This identification can be obtained using the technique analogous to \cite[Proposition 3.1.8]{parabook} for $T(K/L\times \mathbb{R})$ instead of $T(K/L)$. Indeed, the Cartan connection in the background remains the same and we only need to restrict ourselves to normal Weyl connections that coincide with $D$ at $x$ and project the results given by the Cartan connection to $T(K/L\times \mathbb{R})$ instead of $T(K/L)$.
However, $$[X,Y]-[\alpha(X+\fh),\alpha(Y+\fh)]+\alpha([X+\fh,Y+\fh])=[X,Y] \mod \frak{u}(p,q)\oplus\C^{n*} \oplus \R^*$$
 due to the condition (5) from  Lemma \ref{lem4}. Therefore we have $$([\xi,\eta]-[J\xi,J\eta]+J([J\xi,\eta]+[\xi,J\eta]))(x)=[X,Y]-[JX,JY]+J([JX,Y]+[X,JY]).$$
 
Let us now discuss possible values of this expression for all possible incomes:
\begin{itemize}
\item For $X,Y\in \C^n$ we have $[X,Y]-[JX,JY]+J([JX,Y]+[X,JY])=0$.
\item For $X\in \C^n$ and $Y=JZ\in \R$ we have $[X,Y]-[JX,JY]+J([JX,Y]+[X,JY])=[JX,Z]-J([X,Z])=0$.
\item For $X\in \C^n$ and $Y=Z$ we have $[X,Y]-[JX,JY]+J([JX,Y]+[X,JY])=[X,Z]+J([JX,Z])=0$.
\item For $X=JZ\in \R$ and $Y=Z$ we have  $[X,Y]-[JX,JY]+J([JX,Y]+[X,JY])=[JZ,Z]+[Z,JZ]=0$.
\end{itemize}
The remaining possibilities vanish trivially.
 Thus the complex structure is integrable. Then the zero section is a CR embedding, because it is a closed orbit.
\end{proof}

In holomorphic coordinates on $U\subset K/L\times \mathbb{R}$, the hypersurface $K/L\cap U \subset \C^{n+1}$ may be described as a zero set of a function $F: U\to \mathbb{R}$.  Theorem \ref{embed} and  Lemma \ref{dve-symetrie}  provide  distinguished holomorphic coordinates, in which the function $F$ has a specific form.
\begin{cor*}
Let $K$ be the Lie group generated by all symmetries of a non--flat symmetric CR geometry $(M,\mathcal{H},J)$. Suppose that $\Ad(L^0)|_{\fh/\fl}=\Ad(L)|_{\fh/\fl}$. Then for every point $x\in M$, there is a holomorphic coordinate system on $U\subset K/L\times \mathbb{R}$ centred at $x$ such that the function $F(z,w)$ defining $M$ satisfies $F(z,w)=F(-z,w)$.
\end{cor*}

\section{Locally flat CR symmetric spaces} \label{sec6}

Locally flat CR geometries are always locally symmetric (in our sense). Therefore, the following question appear:
Which local symmetries are globally defined?
The answer depends on the topology of the manifold. We show on series of examples that various situations are possible. There are two sources of examples that we study here that are related to flag manifolds. The first series of examples follows the construction from \cite{ja-CEJM,conf} that we apply to CR geometries. The second series of examples involves CR geometries on orbits of real forms in flag manifolds from \cite{AMN}.

\subsection{Non--homogeneous symmetric CR geometries}

Let us apply the construction from \cite{ja-CEJM,conf} to CR geometries. We start with the standard model $PSU(p+1,q+1)/P$. Consider the CR manifold $M:=PSU(p+1,q+1)/P - \{\langle u \rangle,\langle v \rangle\}$, where $u,v \in \C^{n+2}$ are arbitrary non--zero null vectors of $m$. The group $K(u,v)$ of CR transformations of the flat CR geometry on $M$ has two connected components. The identity component  of $K(u,v)$ is the intersection of stabilizers  of $\langle u \rangle$ and  $\langle v \rangle$. The other component contains the elements that swap $\langle u \rangle$ and $\langle v \rangle$. We check whether there is a symmetry at each $K(u,v)$--orbit on $M$. Let us emphasize that if all symmetries at one point of a $K(u,v)$--orbit preserve or swap the points $\langle u\rangle$ and $\langle v\rangle$ then all symmetries at all points of the whole orbit have the same property.
The orbits of the action of $K(u,v)$ on $M$ are characterized by the fact that the action preserves
\begin{itemize}
\item the subspace $\langle u,v\rangle$, and  
\item the (non)--isotropy with respect to the Hermitian form $m$.
\end{itemize}
Moreover, the action of $K(u,v)$ on $\langle u,v\rangle$ depends on whether $\langle u,v\rangle$ is isotropic subspace or not.

\begin{exam} \label{exam1}
Assume that $p,q > 1$, i.e., not the Riemannian signature. Consider the CR manifold $M=PSU(p+1,q+1)/P - \{\langle u \rangle,\langle v \rangle\}$ for arbitrary non--zero null vectors $u,v \in \C^{n+2}$ isotropic with respect to $m$, i.e., $m(u,v)=0$. Then $\langle u,v\rangle -\{\langle u\rangle ,\langle v \rangle \}$ consists of a single orbit of $K(u,v)$. Furthermore, $K(u,v)$--orbits of points $\langle x\rangle $ such that $x \notin \langle u,v\rangle-\{\langle u\rangle ,\langle v \rangle \}$ depend only on the (non)--isotropy of $x$ with respect to $u, v$.

We show that there exist symmetries at all points of each orbit of $K(u,v)$. Instead of fixing $\langle u\rangle,\langle v \rangle$ and discussing symmetries at various points $\langle x \rangle$, we fix the point $\langle x \rangle$ as the point $\langle e_0 \rangle$ given by the first vector of the standard basis $e_0,\dots,e_{n+1}$ of $\C^{n+2}$ and we choose admissible $\langle u \rangle$ and $\langle v \rangle$ such that $\langle e_0 \rangle$ lies in the correct orbit. Then we find all symmetries at $\langle e_0 \rangle$. Let us recall that all symmetries of the standard model at the origin $\langle e_0 \rangle$ are of the form
\begin{align*}
s_{Z,z}=\left(
\begin{matrix}
-1&-Z& iz+\frac12ZIZ^* \\ 0&E&-IZ^* \\ 0&0&-1
\end{matrix}
\right)
,\end{align*}
where $Z=(z_1, \dots, z_n) \in \C^{n*}$ and $z \in \R^*$ are arbitrary, and involutive are those satisfying $z=0$.
\\[1mm]
(1) Let us start with the orbit corresponding to the case $m(e_0,u)\neq 0$ and $m(e_0,v)\neq 0$. We  choose $u=ie_0+\sqrt{2}e_1-ie_{n+1}$ and  $v=ie_0-\sqrt{2}e_{n}+ie_{n+1}$. Direct computation gives that there is exactly one symmetry $s_{Z,z}$, where $Z=(-i\sqrt{2},0, \dots, 0,-i\sqrt{2})$ and $z=0$. This symmetry is involutive and swaps $\langle u \rangle$ and $\langle v \rangle$. There is no symmetry preserving them.
\\[1mm]
(2) Let us now consider the orbit for the case $m(e_0,u)=0$ and $m(e_0,v)\neq 0$ (which is the same orbit as the orbit for the case $m(e_0,u)\neq 0$ and $m(e_0,v)=0$). We  choose $u=e_1+e_{n}$ and $v=ie_{n+1}$. Direct computation gives that there is exactly one symmetry $s_{Z,z}$, where $Z=(0, \dots,0)$ and $z=0$. This symmetry is involutive and preserves $\langle u \rangle$ and $\langle v \rangle$. There is no symmetry swapping them.
\\[1mm]
(3) The next possibility is the orbit for the case $m(e_0,u)=m(e_0,v)=0$ and $e_0 \in \langle u,v \rangle$. We  choose $u=e_{1}+e_{n}$ and $v=e_0+e_{1}+e_{n}$. Computation gives that there are (many) symmetries $s_{Z,z}$, where $Z=(z_1, \dots,z_{n})$ with components $z_k=a_k+ib_k$ for $k=1, \dots, n$ satisfies $a_1+a_{n}+1=0$ and $b_1+b_{n}=0$, and $a_k,b_k$ for $k=2, \dots, n-1$ and $z$ are arbitrary.
All these symmetries swap $\langle u \rangle$ and $\langle v \rangle$, and there are no symmetries preserving them. In particular, there are also non--involutive symmetries for $z\neq 0$.
\\[1mm]
In fact, this covers all possible orbits for the case $p=1$ or $q=1$, i.e., the Lorentzian signature.
In the other cases, there is one more orbit.
\\[1mm]
(4) Consider the orbit for the case $m(u,e_0)=m(v,e_0)=0$ and $e_0 \notin \langle u,v \rangle$. We choose $u=e_1+e_{n}$ and $v=e_2+e_{n-1}$. Computation gives that there are (many) symmetries $s_{Z,z}$, where $Z=(z_1, \dots,z_{n})$ satisfies $a_1+a_n=0$, $b_1+b_n=0$, $a_2+a_{n-1}=0$ and $b_2+b_{n-1}=0$ and  $a_k,b_k$ for $k=3, \dots, n-2$ and $z$ are arbitrary. All these symmetries preserve $\langle u \rangle$ and $\langle v \rangle$ and there are no symmetries swapping them. In particular, there are also non--involutive symmetries for $z\neq 0$.
\\[1mm]
Altogether, symmetries at different orbits behave differently. Therefore, there is no smooth system of symmetries. In particular, there is no pseudo--Riemannian metric compatible with the CR geometry that would be preserved by some symmetry at every point.
\hspace{\fill} $\Diamond$
\end{exam}

Let us show that this principle does not work if we remove two points corresponding to non--isotropic vectors. 

\begin{exam}
Consider the manifold $M=PSU(p+1,q+1)/P - \{\langle u \rangle,\langle v\rangle\}$ for arbitrary non--zero null vectors $u,v \in \C^{n+2}$ that are non--isotropic for $m$, i.e. $m(u,v) \neq 0$.
We choose $u=e_{n+1}$ and $v=e_0+\sqrt{2}e_1+(1+i)e_{n}$.  Computation gives that there is no symmetry at $\langle e_0\rangle$ preserving or swapping $\langle u \rangle$ and $\langle v\rangle$. Let us remark that the component of identity of $K(u,v)$ is isomorphic to the group $CSU(p,q)$ and $K(u,v)$ does not act transitively on $\langle u,v\rangle-\{\langle u\rangle ,\langle v \rangle \}$.
\hspace{\fill} $\Diamond$
\end{exam}

\subsection{Flat homogeneous symmetric CR geometries and orbits of real forms in complex flag manifolds}

It follows from Lemma \ref{lem4} that an extension $(\alpha,i)$ of $(K,L)$ to $(PSU(p+1,q+1),P)$ corresponds to  flat CR geometry if and only if $\alpha$ is a Lie algebra homomorphisms. Therefore, we can present examples of extensions describing flat homogeneous symmetric CR geometries just by specifying the Lie subalgebra of $\frak{su}(p+1,q+1)$ that coincides with the image of $\alpha$. In general, the group $K$ does not have to contain symmetries. Moreover, symmetries do not have to preserve $\alpha(\fk)$. This is satisfied if $K$ is the group generated by symmetries or the full group of CR automorphisms. 

\begin{exam}
Consider the orbits of $PSp(1,1)$ on $\mathbb{C}P^4$ given by inclusion $PSp(1,1)\subset PSp(4,\mathbb{C})\subset PGl(4,\mathbb{C})$. Due to the isomorphisms $PSp(1,1)\cong PO(1,4)$, these orbits can also be interpreted as orbits in the flag manifold of 2--planes in quadric in $\mathbb{C}P^5$. There is a normal extension given by identifying the following Lie subalgebra of $\frak{su}(2,2)$ with the image of $\alpha(\frak{sp}(1,1)):$

$$
\left( \begin{matrix} l_1+il_2&i X_2+i 
l_5- X_1+l_4&iX_4+il_5-X_3+
l_4&i \left(l_3+x \right) \\ iX_2+X_1&-{i \over 2} \left( 2l_2+2x+ l_3 \right) &-{il_3 \over 2
}-l_1&iX_2+il_5+X_1-l_4
\\ iX_4+ X_3&{il_3 \over 2}-l_1
&-{i \over 2} \left( 2l_2-2x-l_3 \right) &-iX_4-i
l_5-X_3+l_4\\ ix&iX_2-X_1&X_3-iX_4&- l_1+il_2\end{matrix}
 \right),
 $$
 where $l_i$--entries generate the Lie algebra of the stabilizer $L=CSO(2)\rtimes S^2\mathbb{R}^2$ of a point in the minimal orbit. Precisely, $\langle l_1,l_2\rangle=\frak{cso}(2)$ and $\langle l_3,l_4,l_5\rangle=S^2\mathbb{R}^2$.
\end{exam}

\begin{exam}
Consider the orbits of $PSp(4,\mathbb{R})$ on $\mathbb{C}P^4$ given by inclusion $PSp(4,\mathbb{R})\subset PSp(4,\mathbb{C})\subset PGl(4,\mathbb{C})$. Due to the isomorphisms $PSp(4,\mathbb{R})\cong PO(2,3)$, these orbits can again be interpreted as orbits in the flag manifold of 2--planes in quadric in $\mathbb{C}P^5$. There is a normal extension given by identifying the following Lie subalgebra of $\frak{su}(2,2)$ with the image of $\alpha(\frak{sp}(n+2,\mathbb{R})):$

$$
 \left( \begin {matrix} l_1+i l_2&X_1-i
X_2+l_4+il_5&X_3-i X_4-l_4-i 
l_5&i \left(l_3+x \right) \\ i X_2+ X_1&-{i \over 2} \left( 2 l_2-2x- l_3 \right) & 
l_1-{il_3 \over 2}&-iX_2+i l_5- X_1- l_4
\\ iX_4+ X_3& l_1+{il_3 \over 2}
&-{i \over 2} \left( 2l_2+2x+l_3 \right) &iX_4+il_5+X_3-l_4\\ ix&iX_2-X_1&X_3-iX_4&- l_1+il_2\end{matrix}
 \right),
 $$
 where $l_i$--entries generate the Lie algebra of the stabilizer $L=CSO(2)\rtimes S^2\mathbb{R}^2$ of a point in $5$--dimensional orbit (which is not minimal). Precisely, $\langle l_1,l_2\rangle=\frak{cso}(2)$ and $\langle l_3,l_4,l_5\rangle=S^2\mathbb{R}^2$.
\end{exam}

In both examples, $\fk$  is simple and $\fq$ is a parabolic subgalgebra of $\fk_\mathbb{C}$. In \cite{AMN}, the authors discuss which CR algebras $(\fk,\fq)$ for simple Lie algebras $\fk$ and  parabolic subalgebras $\fq$ of the complexification of $\fk$ are symmetric. In fact, they correspond to orbits of real forms in complex flag varieties. Therefore, symmetric CR algebras of these types generalize bounded symmetric domains.

We show that if CR algebra $(\fk,\fq)$ for a simple Lie algebra $\fk$ and a parabolic subalgebra $\fq$ of the complexification $\fk_{\mathbb{C}}$ of $\fk$ corresponds to non--degenerate symmetric CR geometry of hypersurface type, then the geometry is necessarily flat.  Therefore, we can use the results of \cite{O} to classify all possible cases.

\begin{prop*}
Let $(\fk,\fq)$ be a CR algebra such that $\fk$ is simple and $\fq$ is a parabolic subalgebra of $\fk_{\mathbb{C}}$ and the corresponding CR geometry is non--degenerate and of hypersurface type. Then the following statements hold:

\begin{enumerate}

\item If the CR geometry is symmetric, then the CR geometry is flat.

\item If the CR geometry is flat, then it corresponds to one of the following possibilities:

\begin{enumerate}

\item $\fk=\frak{su}(p+1,q+1)$ and $\fl=\fp$,

\item $\fk=\frak{sp}({p+1\over 2},{q+1\over 2})$ and $\fl=\frak{co}(2)\oplus \frak{sp}({p-1 \over 2},{q-1\over 2}) \oplus \mathbb{R}^{2}\otimes\mathbb{R}^{n-2*}\oplus S^2\mathbb{R}^2$,

\item $\fk=\frak{sp}(n+2,\mathbb{R})$ and $\fl=\frak{co}(2)\oplus \frak{sp}(n-2,\mathbb{R}) \oplus \mathbb{R}^{2}\otimes\mathbb{R}^{n-2*}\oplus S^2\mathbb{R}^2$, where $\mathbb{R}^{2}\otimes\mathbb{R}^{n-2*}\oplus S^2\mathbb{R}^2$ is the positive part of the parabolic subalgebra corresponding to the stabilizer of a Lagrangian 2--plane in $\mathbb{R}^{n+2}$.

\end{enumerate}

\item If the CR geometry is flat and $\fk_{\mathbb{C}}=\frak{sp}(n+2,\mathbb{C})$ is the full Lie algebra of complete infinitesimal automorphism and $n>2$, then the corresponding CR geometry is not symmetric.

\item If the CR geometry is flat and corresponds to $(2n+1)$--dimensional orbit of the real form of $\frak{sp}(n+2,\mathbb{C})$ in $\mathbb{C}P^{n+2}$, then the corresponding CR geometry is symmetric if and only if $n=2$ or the orbit is minimal, i.e., if $\fk\neq \frak{sp}(n+2,\mathbb{R})$.

\end{enumerate}
\end{prop*}
\begin{proof}
If such symmetric CR geometry is non--flat, then $K$ has to be generated by symmetries and it follows from \cite[Theorem 3.1]{HG-dga} that the complexification of $\fl$ does not contain a Cartan subalgebra of $\fk_{\mathbb{C}}$. 
On the other hand,  if $\fk$ is simple and $\fq$ is a parabolic subalgebra of $\fk_{\mathbb{C}}$, then $\fq\cap \bar{\fq}$ contains a Cartan subalgebra of $\fk_{\mathbb{C}}$. This is a contradiction and therefore, the claim (1) holds.

If such symmetric CR geometry is flat, then $\fk_{\mathbb{C}}$ is isomorphic to a Lie subalgebra of $\frak{sl}(n+2,\mathbb{C})$, $\fq=\fk_\mathbb{C}\cap (\mathbb{C}^*\oplus \fp_{\mathbb{C}})$ is a parabolic subalgebra of $\fk_\mathbb{C}$ and $\fk_{\mathbb{C}}/\fq=\frak{sl}(n+2,\mathbb{C})/ (\mathbb{C}^*\oplus \fp_{\mathbb{C}}).$ All such cases are classified in \cite{O} and it follows that $\fk_{\mathbb{C}}=\frak{sl}(n+2,\mathbb{C})$ or $\fk_{\mathbb{C}}=\frak{sp}(2n+2,\mathbb{C})$. The first case corresponds to the standard model. The remaining cases correspond to the symmetric pair $(\frak{sl}(n+2,\mathbb{C}), \frak{sp}(2n+2,\mathbb{C}))$. Real forms of this symmetric pair are well--known and correspond to suitable inclusions $\frak{sp}({p+1\over 2},{q+1\over 2})\subset \frak{su}(p+1,q+1)$ or $\frak{sp}(n+2,\mathbb{R})\subset \frak{su}(n+1,n+1)$. If such inclusion provides an extension, then it is unique (up to equivalence). Therefore, it suffices to show that the cases in the claim (2) correspond to non--degenerate CR geometries of hypersurface type.  This follows from the fact that $\frak{co}(2)\cong \mathbb{C}$ defines a complex structure on the whole $\fk/\fl$ with the exception of the trace part of $(S^2\mathbb{R}^2)^*$.
 
If
$\fk_{\mathbb{C}}=\frak{sp}(n+2,\mathbb{C})$ is the full Lie algebra of complete infinitesimal automorphism and the corresponding CR geometry is symmetric, then $\Ad(s_{0,0})$ induces an involution of $\frak{sp}(n+2,\mathbb{R})$.  It follows from the description of $\fl$ that the stabilizer has to have the form $\frak{gl}(2,\mathbb{C})\oplus \frak{sp}(n-2,\mathbb{C})$. Therefore the claim (3) follows from the fact that this stabilizer does not appear in the classification of simple symmetric spaces if $n>2$.

Since $\frak{sp}(n+2,\mathbb{C})$ is maximal subalgebra of $\frak{sl}(n+2,\mathbb{C})$, the only possibility for the orbit to be symmetric is to be equivalent to standard model which is compact. Since the orbit is compact if only if the orbit is minimal, the claim (4) follows. It follows from \cite{AMN} that the orbit is minimal if and only if $\fk\neq \frak{sp}(n+2,\mathbb{R})$.
\end{proof}

\end{document}